\documentclass[12pt]{amsart}
\usepackage{latexsym}
\usepackage{amsthm}
\usepackage{amsmath}
\usepackage{amsfonts}
\usepackage{amssymb}
\usepackage[dvips]{graphicx}
\usepackage{xypic}
\input xy
\xyoption{all}

\newtheorem{theo}{Theorem}[section]
\newtheorem{lemma}[theo]{Lemma}
\newtheorem{propo}[theo]{Proposition}
\newtheorem{defi}[theo]{Definition}
\newtheorem{coro}[theo]{Corollary}
\newtheorem{rem}[theo]{Remark}
\newtheorem{exam}[theo]{Example}

\newcommand\id{\operatorname{id}}
\newcommand\Id{\operatorname{Id}}

\newcommand\Ab{\operatorname{\bf Ab}}

\newcommand\Alg{\operatorname{Alg}}
\newcommand\Coalg{\operatorname{Coalg}}

\newcommand\dom{\operatorname{dom}}
\newcommand\cod{\operatorname{cod}}
\newcommand\comp{\operatorname{comp}}

\newcommand\cc{\mathcal {C}}

\newcommand\cf{\mathcal {F}}

\newcommand\cj{\mathcal {J}}
\newcommand\ck{\mathcal {K}}
\newcommand\cl{\mathcal {L}}

\newcommand\crr{\mathcal {R}}

\newcommand\cp{\mathcal {P}}
\newcommand\cw{\mathcal {W}}
\newcommand\cx{\mathcal {X}}

\date{July 17, 2015}
 
\begin{document}
\title[Accessible model categories]
{Accessible model categories}
\author[J. Rosick\'{y}]
{J. Rosick\'{y}$^{*}$}
\thanks{$^{*}$  Supported by the Grant agency of the Czech republic under the grants 201/11/0528 and P201/12/G028.} 
\address{
\newline 
Department of Mathematics and Statistics\newline
Masaryk University, Faculty of Sciences\newline
Kotl\'{a}\v{r}sk\'{a} 2, 611 37 Brno, Czech Republic\newline
rosicky@math.muni.cz
}
 
\begin{abstract}
We prove that a weak factorization system on a locally presentable category is accessible if and only if it is small generated in the sense
of R. Garner. Moreover, we discuss an analogy of Smith's theorem for accessible model categories.
\end{abstract} 
\subjclass{ }

\maketitle

\section{Introduction}
A weak factorization system on a category $\ck$ consists of two classes of morphisms $\cl$ and $\crr$ related by the lifting property
(i.e., $\crr=\cl^\square$ and $\cl={}^\square\crr$) and such that any morphism $f$ of $\ck$ has a factorization $f=Rf\cdot Lf$ with $Lf\in\cl$ 
and $Rf\in\crr$. Mostly, this factorization is functorial -- giving functors $L,R:\ck^2\to\ck^2$ on the category of $\ck$-morphisms. In particular, 
it happens when $\ck$ is locally presentable and $\cl$ is generated by a set of morphisms. In this case, we get an accessible functorial factorization 
in the sense that $L$ and $R$ are accessible functors. But there are important examples of accessible weak factorization systems on a locally presentable
category which are not generated in this sense. Our example is on abelian groups when $\cl$ consists of split monomorphisms (which can be extended 
to $R$-modules over a ring $R$ where $\lambda$-pure epimorphisms never coincide with split epimorphisms). But there are more examples (e.g., 
\cite{Ri} 4.4). We will show that accessible weak factorization systems on a locally presentable category are precisely those which are generated 
by a small category $\cc$ of morphisms in the sense of \cite{G}. This makes possible to define $\cc^\boxplus$ by using coherent liftings. R. Garner 
\cite{G} introduced this concept in order to understand natural weak factorization systems of M. Grandis and W. Tholen \cite{GT}. Here, $L$ is a comonad,
$R$ is a monad and $\boxplus$ relates algebras for a monad $R$ with coalgebras for a comonad $L$. We stay in the context of functorial weak factorization 
systems where $L$ is only copointed and $R$ is pointed and use $\boxplus$ to relate $L$-coalgebras with $R$-algebras. Then the main result of \cite{G}
that a small category $\cc$ generates a natural weak factorization systems makes possible to obtain our characterization of accessible weak factorization
systems.

An accessible model category is a model category on a locally presentable category $\ck$ where the both weak factorization systems are accessible.
A natural question is when an accessible weak factorization system $(\cl,\crr_0)$ on $\ck$ and a class $\cw$ yield a model category where $\cl$ 
are cofibrations and $\cw$ are weak equivalences. If $(\cl,\crr_0)$ is cofibrantly generated then the answer is given by Smith's theorem (see \cite{B})
and the resulting model category is combinatorial, which means that the second weak factorization system is cofibrantly generated as well.  

The author's interest in accessible weak factorization systems which are not cofibrantly generated was refreshed by a discussion with Emily Riehl
in July 2014. In particular, she wished Smith's theorem for algebraic model categories. Recent work \cite{BG} introduced generation by a small double
category and proved that an algebraic weak factorization system is generated in this sense if and only if it is accessible. Note that an algebraic
factorization system $(L,R)$ does not need to be a weak factorization system because the classes of morphisms underlying coalgebras for the comonad $L$
and algebras for the monad $R$ are not necessarily closed under retracts. By taking the closure under retracts, one gets the underlying weak factorization
system $(L,R)$.

\section{Functorial weak factorization systems}
Let $\ck$ be a category and $f: A\to B$, $g: C\to D$ morphisms
such that in each commutative square
$$
\xymatrix@=3pc{
A \ar[r]^{u} \ar[d]_{f}& C \ar[d]^g\\
B\ar[r]_v & D
}
$$
there is a diagonal $d:B \to C$ with $df=u$ and $gd=v$.
Then we say that $g$ has the \textit{right lifting property}
w.r.t. $f$ and $f$ has the \textit{left lifting property} w.r.t.
$g$. 
For a class $\cx$ of morphisms of $\ck$ we put
\begin{align*}
\cx^{\square}& = \{g|g \ \mbox{has the right lifting property
w.r.t.\ each $f\in \cx$\} and}\\
{}^\square\cx & = \{ f|f \ \mbox{has the left lifting property
w.r.t.\ each $g\in \cx$\}.}
\end{align*}

\begin{defi}\label{def2.1}
{\em    
A \textit{weak factorization system} $(\cl,\crr)$ in a
category $\ck$ consists of two classes $\cl$ and $\crr$ of morphisms
of $\ck$ such that
\begin{enumerate}
\item[(1)] $\crr = \cl^{\square}$, $\cl = {}^\square \crr$, and
\item[(2)] any morphism $h$ of $\ck$ has a factorization $h=gf$ with
$f\in \cl$ and $g\in \crr$.
\end{enumerate}

A weak factorization system $(\cl,\crr)$ is called \textit{cofibrantly generated} if there is a set
$\cx$ of morphisms such that $\crr=\cx^\square$.
}
\end{defi}

There are various definitions of a functorial weak factorization system, we will follow \cite{Ri}. In what follows,
$\ck^2$ denotes the category of morphisms and $\ck^3$ the category of composable pairs of morphisms of $\ck$.
There are the domain, the codomain and the composition functors $\dom:\ck^2\to\ck$, $\cod:\ck^2\to\ck$ and $\comp:\ck^3\to\ck^2$.

\begin{defi}\label{def2.2}
{
\em
A \textit{functorial factorization} is a functor $E:\ck^2\to\ck^3$ such that
$$
\comp\cdot E=\Id.
$$
}
\end{defi}

\begin{rem}\label{re2.3}
{
\em
(1) Let $p_1:\ck^3\to\ck^2$ and $p_2:\ck^3\to\ck^2$ be the projections sending the composable pair to its first and the second morphism.
Then $E$ determines a pair of functors $L=p_1\cdot E:\ck^2\to\ck^2$ and $R=p_2\cdot E:\ck^2\to\ck^2$ and the functorial factorization of $f$ 
is $f=Rf\cdot Lf$. The middle object of the composable pair $Ef$ will be denoted as $E_0f$. With this notation, the functorial factorization $E$ 
sends a commutative square
$$
\xymatrix@=3pc{
A \ar[r]^{u} \ar[d]_{f}& C \ar[d]^g\\
B\ar[r]_v & D
}
$$
to a commutative rectangle
$$
\xymatrix@=3pc{
A \ar[r]^{u} \ar[d]_{Lf} & C\ar[d]^{Lg}\\
E_0f \ar[r]^{E(u,v)} \ar[d]_{Rf} & E_0g
\ar[d]^{Rg}\\
B\ar[r]_{v} &D
}
$$ 

(2) The functor $L$ is \textit{copointed}, the natural transformation $\varepsilon:L\to\Id$ has components $\varepsilon_f=(\id,Rf)$. Dually,
$R$ is \textit{pointed} by $\rho:\Id\to R$, $\rho_f=(Lf,\id)$. An $L$-\textit{coalgebra} is a pair $(f,(u,v))$ where $f$ is a $\ck$-morphism 
and $(u,v):f\to Lf$ is a morphism in $\ck^2$ such that $\varepsilon_f(u,v)=\id_f$. Thus an $L$-coalgebra reduces to a pair $(f,s)$ such that
$sf=Lf$ and $Rf\cdot s=\id_B$. Morphisms of $L$-coalgebras $(f,s)\to (f',s')$ are $\ck^2$-morphism $(x,y):f\to f'$ such that $E(x,y)s=s'y$.
The category $\Coalg(L)$ of $L$-coalgebras is equipped with the forgetful functor $\Coalg(L)\to\ck^2$ whose image is $\cl$. In fact, given
$f\in\cl$ then $s$ is obtained by the lifting of $f$ with respect to $Rf$.
Dually, we have the category $\Alg(R)$ of $R$-\textit{algebras} and $\Alg(R)\to\ck^2$ whose image is $\crr$.
In particular, an $R$-algebra reduces to a pair $(g,t)$ such that $gt=Rg$ and $tLg=\id$. The forgetful functors are faithful, which means 
that $\Coalg(L)$ and $\Alg(R)$ are \textit{concrete categories} over $\ck^2$. 

If $\ck$ is a cocomplete category then $\Coalg(L)$ is cocomplete and the forgetful functor to $\ck^2$ preserves colimits. Dually, if $\ck$ is complete,
$\Alg(R)$ is complete and the forgetful functor preserves limits.

(3) Having a general concrete category $\cx$ over $\ck^2$, the underlying morphism of $X\in\cx$ will be denoted as $|X|$. Morphisms $X\to Y$
of $\cx$ will be identified with the underlying morphisms $(u,v):|X|\to |Y|$.

(4) Given $f:A\to B$, consider the morphism $\varepsilon_f:Lf\to f$
$$
\xymatrix@=3pc{
A \ar[r]^{\id} \ar[d]_{LLf} & A\ar[d]^{Lf}\\
E_0Lf \ar[r]^{E(\id,Rf)} \ar[d]_{RLf} & E_0f
\ar[d]^{Rf}\\
E_0f\ar[r]_{Rf} &B
}
$$ 
}
\end{rem}

\begin{defi}\label{def2.4}
{
\em
A weak factorization system $(\cl,\crr)$ is called \textit{functorial} if it has a functorial factorization with $Lf\in\cl$ and $Rf\in\crr$
for all $f$.
}
\end{defi}

Of course, a weak factorization system can have many functorial factorizations.

\begin{defi}\label{def2.5}
{
\em
A weak factorization system will be called \textit{left accurate} if it has a functorial factorization such that $E(\id,Rf)=RLf$ for each $f$ 
in $\ck^2$. Dually, we define \textit{right accurate} functorial weak factorization systems. A functorial weak factorization system is \textit{accurate} 
if it is both left and right accurate.
}
\end{defi}

\begin{lemma}\label{le2.6}
Any cofibrantly generated weak factorization system in a locally presentable category is functorial and left accurate.
\end{lemma}
\begin{proof}
We can assume that $\ck$ be locally $\lambda$-presentable and $(\cl,\crr)$ is cofibrantly generated by a set $\cx$ of morphisms
having $\lambda$-presentable domains and codomains. For $f$ in $\ck^2$, consider the diagram consisting of all spans
$$
\xymatrix@=3pc{
A &\\
X \ar [u]^{u} \ar [r]_{h} & Y
}
$$
such that $h\in\cx$ and there exists $v:Y\to B$  with $vh=fu$. Let $E_{00}f$ be a colimit of this diagram and $\rho_0:\Id\to R_0$ be the pointed functor 
$R_0:\ck^2\to\ck^2$ such that $(\rho_0)_f:A\to E_{00}f$ is the component of a colimit cocone and $R_0f:E_{00}f\to B$ be the induced morphism. 
Consider the chain of iterations
$$
\Id \xrightarrow{\quad\rho_0\quad} R_0
             \xrightarrow{\quad\rho_0R_0\quad} R_0^2\xrightarrow{\quad\rho_0R_0^2\quad}\dots 
$$
where we take colimits in limit steps. We get morphisms $R_0^\alpha f:E_{0\alpha}f\to B$.
Let $R=R_0^\lambda$ and $E_0f=E_{0\lambda}f$ be the domain of $Rf$. Let $Lf:A\to E_0f$ be $f$-component of $\Id\to R$.
Since any morphism $u:X\to E_0f$, where $h:X\to Y$ is in $\cx$, factorizes through some $E_{0\alpha}$, $\alpha<\lambda$, $(L,R)$ is a functorial factorization 
of $(\cl,\crr)$. Moreover, since both $E(\id,Rf)$ and $RLf$ are induced morphisms from a colimit, $(L,R)$ is left accurate. 
\end{proof}

\begin{rem}\label{re2.7}
{
\em
The functorial factorization from \ref{le2.6} corresponds to the usual small object argument (see \cite{AHRT1} or \cite{G}).
This functorial factorization is not right accurate because at steps $\alpha^+$, $\alpha>0$, one adds new solutions to old lifting
problems and $E(L_\alpha f,\id)$ sends old solutions to new ones while $L_\alpha R_\alpha f$ sends them to old ones. This is repaired
by Garner's modification of the small object argument which adds coequalizers of $E(L_\alpha f,\id)$ and $L_\alpha R_\alpha f$ (see
\cite{G} 6.5). The resulting functorial factorization remains left accurate but is right accurate as well. Thus any cofibrantly
generated weak factorization system in a locally presentable category is accurate.
}
\end{rem}



\begin{exam}\label{ex2.8}
{
\em
Let $\Ab$ be the category of abelian groups. Let $\cl$ consist of split monomorphisms and $\crr$ of split epimorphisms.
Following \cite{RT} 2.7, $(\cl,\crr)$ is a weak factorization system in $\Ab$. This factorization system is functorial:
given $f:A\to B$, we put $Lf=\langle\id_A,f\rangle:A\to A\times B$ and $Rf=p_2:A\times B\to B$ (see \cite{AHRT} 1.6). We will show
that $(\cl,\crr)$ is not cofibrantly generated.

Assume that $(\cl,\crr)$ is cofibrantly generated. Then there is a regular cardinal $\lambda$ such that split epimorphisms
are closed under $\lambda$-filtered colimits. This means that $\lambda$-pure epimorphisms split (see \cite {AR1} Proposition 3).
Consequently, each abelian group is a retract of a coproduct of $\lambda$-presentable abelian groups, which contradicts
\cite{CH} 5.10.

It is easy to see that this functorial factorization system is accurate.
}
\end{exam}
 
\section{Small generated weak factorization systems}
The following concept was introduced in \cite{G}.

\begin{defi}\label{def3.1}
{
\em
If $\cx$ is a concrete category over $\ck^2$ then $\cx^\boxplus$ will be a category whose objects are pairs $(g,\phi)$  where $g$ is
a $\ck$-morphism and $\phi$ is a \textit{lifting function} that assigns a diagonal $\phi(X,u,v)$ to each commutative square
$$
\xymatrix@=4pc{
A \ar[r]^{u} \ar[d]_{|X|}& C \ar[d]^g\\
B\ar[r]_v & D
}
$$
with $X$ in $\cx$. We require that $\phi$ is coherent with respect to $\cx$-morphisms in the sense that, for any $\cx$-morphism $(x,y):X_1\to X_2$,
we have
$$
\phi(X_1,ux,vy)=\phi(X_2,u,v)y.
$$

Morphisms $(z,t):(g_1,\phi_1)\to (g_2,\phi_2)$ are morphisms $(z,t):g_1\to g_2$ such that
$$
z\phi_1(X,u,v)=\phi_2(X,zu,tv)
$$
for each $(u,v):|X|\to g_1$. The forgetful functor $\cx^\boxplus\to\ck^2$ forgets lifting functions. Thus $\cx^\boxplus$ is a concrete category
over $\ck^2$.

Dually, we define ${}^\boxplus\cx$.
} 
\end{defi}

\begin{rem}\label{re3.2}
{
\em
(1) Let $\cx$ be a concrete category over $\ck^2$. Since $\boxplus$ is a Galois correspondence, we have
$$
\cx\subseteq {}^\boxplus(\cx^\boxplus)   
$$
and
$$
({}^\boxplus(\cx^\boxplus))^\boxplus=\cx^\boxplus.
$$
Moreover,
$$
|\cx^\boxplus|\subseteq |\cx|^\square.
$$

(2) Let $H:\cx_1\to\cx_2$ be a \textit{concrete functor}, i.e., a functor $H$ commuting with the forgetful functors. Then we get
a concrete functor
$$
H^\boxplus:\cx_2^\boxplus\to\cx_1^\boxplus
$$
such that
$$
H^\boxplus(g,\phi)=(g,\phi(HX,u,v)).
$$

In particular, given a functorial weak factorization system $(\cl,\crr)$, then any subcategory $\cc$ of $\Coalg(L)$ induces the functor 
$\Coalg(L)^\boxplus\to\cc^\boxplus$. This functor is faithful and it is an isomorphism provided that $\ck$ is cocomplete and $\cc$ is 
\textit{colimit}-\textit{dense} in $\Coalg(L)$, which means that any $L$-coalgebra is a colimit of a diagram in $\cc$.

(3) We have a concrete functor 
$$
\Gamma:\Alg(R)\to\Coalg(L)^\boxplus
$$
assigning to an $R$-algebra $(g,t)$ an object $(g,\phi)$ where 
$$
\phi((f,s),u,v)=tE(u,v)s
$$
for each $L$-coalgebra $(f,s)$.   

If $\ck^2$ is complete then $\Gamma$ preserves limits.

(4) Since
$$
\crr=|\Alg(R)|\subseteq|\Coalg(L)^\boxplus|\subseteq |\Coalg(L)|^\square=\cl^\square=\crr,
$$ 
we have
$$
|\Coalg(L)^\boxplus|=|\Coalg(L)|^\square.
$$
(5) Dually, we have a concrete functor
$$
\Delta:\Coalg(L)\to {}^\boxplus\Alg(R).
$$
and
$$
|^\boxplus\Alg(R)|={}^\square|\Alg(R)|.
$$
{}
}
\end{rem}

\begin{lemma}\label{le3.3}
Let $(\cl,\crr)$ be a left accurate weak factorization system. Then $\Gamma$ is a full embedding.
\end{lemma}
\begin{proof}
As a concrete functor, $\Gamma$ is faithful. Consider a morphism $(x,y):\Gamma(g,t)\to \Gamma(g',t')$. This means that for any $L$-coalgebra
$(f,s)$ and any morphism $(u,v):f\to g$ in $\ck^2$, we have $xtE(u,v)s=t'E(xu,yv)s$. Choose $f=Lg$, $u=\id$ and $v=Rg$. Since $(L,R)$ is left accurate,
$E(u,v)s=RLg\cdot s=\id$. Hence $xt=t'E(x,y)$, which means that $(x,y)$ is a morphism of $R$-algebras. Thus $\Gamma$ is full.

Finally, assume that $\Gamma(g,t)=\Gamma(g',t')$. Then $(\id,\id):(g,t)\to(g',t')$ is a morphism of $R$-algebras. Thus $g=g'$. Since
$(\cl,\crr)$ is left accurate, $E(\id,\id)=\id$ and thus $t=t'$. Hence $\Gamma$ is a full embedding.
\end{proof}





\begin{defi}\label{def3.4}
{
\em
We say that a functorial weak factorization system $(\cl,\crr)$ is \textit{ small generated} if there is a small subcategory 
$\cc$ of $\Coalg(L)$ such that $\Coalg(L)^\boxplus\to\cc^\boxplus$ is an isomorphism.
}
\end{defi}



\section{Accessible weak factorization systems}

\begin{defi}\label{def4.1}
{
\em
A weak factorization system in a locally presentable category is called \textit{accessible} if it has a functorial factorization such that $E$ is an
accessible functor.
}
\end{defi}

\begin{rem}\label{re4.2}
{
\em 
(1) If $(\cl,\crr)$ is an accessible weak factorization system then $L$ and $R$ are accessible functors. Hence the categories $\Coalg(L)$
and $\Alg(R)$ are locally presentable and accessibly embedded in $\ck^2$ (cf. \cite{AR} 2.78).

(2) Any cofibrantly generated weak factorization system in a locally presentable category is accessible.

(3) The weak factorization system from \ref{ex2.8} is accessible.
}
\end{rem}

\begin{theo}\label{th4.3}
A functorial weak factorization system in a locally presentable category is accessible if and only if it is small generated.
\end{theo}
\begin{proof}
Let $(\cl,\crr)$ be an accessible weak factorization system. Following \ref{re4.2}(1), $\Coalg(L)$ is locally presentable and thus it contains
a small dense subcategory $\cc$. Hence $\Coalg(L)^\boxplus=\cc^\boxplus$ (see \ref{re3.2}(2)). Therefore $(\cl,\crr)$ is small generated.

Let $(\cl,\crr)$ be a small generated weak factorization system and $(\overline{L},\overline{R})$ be a free natural weak factorization system 
on $\cc$ (see \cite{G}, 4.4). Thus $(\overline{L},\overline{R})$ is a functorial factorization where $\overline{L}$ is a comonad and $\overline{R}$ 
is a monad. Since $(\overline{L},\overline{R})$ is algebraically free (see \cite{G}, 4.4)), there is a concrete isomorphism
$$
\overline{\Alg}(\overline{R})\to\overline{\Coalg}(\overline{L})^\boxplus\to\cc^\boxplus
$$
where $\overline{\Alg}(\overline{R})$ is the category of algebras over a monad $\overline{R}$ and $\overline{\Coalg}(\overline{L})$ is
the category of coalgebras over a comonad $\overline{L}$. Thus
$$
|\overline{\Alg}(\overline{R})|=|\cc^\boxplus|=|\Coalg(L)^\boxplus|=\crr
$$
(see \ref{re3.2}(4)). Following \cite{G} 2.17, the underlying weak factorization system of $(\overline{L},\overline{R})$ is $(\cl,\crr)$. 
Since $(\overline{L},\overline{R})$ is accessible (see the proof of \cite{G} 4.4), $(\cl,\crr)$ is accessible.
\end{proof}


\begin{coro}\label{cor4.4}
Any cofibrantly generated weak factorization system in a locally presentable category is small generated.
\end{coro}

\begin{rem}\label{re4.5}
{
\em
(1) Following the proof of \ref{th4.3}, any accessible weak factorization system underlies a natural weak factorization system.

(2)  Any factorization system $(\cl,\crr)$ is a functorial weak factorization system. Since any $f\in\cl$ carries a unique $L$-coalgebra
and liftings in \ref{def3.1} are unique, $(\cl,\crr)$ is small generated if and only if it is cofibrantly generated. Thus a factorization
system is accessible if and only if it is cofibrantly generated.
}
\end{rem}

\section{Accessible model categories}
A \textit{model category} is a complete and cocomplete category $\ck$ together with three classes of morphisms $\cf$, $\cc$ and $\cw$ called 
\textit{fibrations}, \textit{cofibrations} and \textit{weak equivalences} such that
\begin{enumerate}
\item[(1)] $\cw$ has the 2-out-of-3 property, i.e., with any two of $f$, $g$, $gf$ belonging to $\cw$ also
the third morphism belongs to $\cw$, and $\cw$ is closed under retracts in the arrow category $\ck^\to$, and
\item[(2)] $(\cc,\cf\cap\cw)$ and $(\cc\cap\cw,\cf)$ are weak factorization systems.
\end{enumerate}
 
Morphisms from $\cf\cap\cw$ are called \textit{trivial fibrations} while morphisms from $\cc\cap\cw$ \textit{trivial cofibrations}. A model category 
$\ck$ is called \textit{combinatorial} provided that $\ck$ is locally presentable and the both weak factorization systems $(\cc,\cf\cap\cw)$ 
and $(\cc\cap\cw,\cf)$ are cofibrantly generated.

\begin{defi}\label{def5.1}
{
\em
A model category $\ck$ is called \textit{accessible} provided that $\ck$ is locally presentable and the both 
weak factorization systems $(\cc,\cf\cap\cw)$ and $(\cc\cap\cw,\cf)$ are accessible.
}
\end{defi}

\begin{rem}\label{re5.2}
{
\em
(1) Any combinatorial model category is accessible.

(2) In an accessible model category, the full subcategory $\cw$ of $\ck^2$ is accessible and accessibly embedded in $\ck^2$. In fact,
$\cw$ is given by the pullback
\[
\xymatrix@=4pc{
\ck^2 \ar[r]^{R} & \ck^2 \\
\cw \ar[u] \ar[r] &
\cf\cap\cw \ar[u]
}
\]
where $R$ belongs to the weak factorization system $(\cc\cap\cw,\cf)$.

(3) Given a functorial weak factorization system $(\cl,\crr)$ and a class of morphisms $\cw$ in $\ck$, we can form the pullback 
\[
\xymatrix@=4pc{
\cw \ar[r]^{} & \ck^2 \\
\cw\times_{\ck^2}\Coalg(L) \ar[u] \ar[r] &
\Coalg(L) \ar[u]
}
\] 
}
\end{rem}

\begin{propo}\label{prop5.3}
Let $(\cl,\crr_0)$ be an accessible weak factorization system in a locally presentable category $\ck$. Then $\cc=\cl$ and $\cw$ make $\ck$
an accessible model category provided that 
\begin{enumerate}
\item[(1)] $\cw$ has the 2-out-of-3 property,
\item[(2)] $\crr_0\subseteq\cw$,
\item[(3)] ${}^\square((\cl\cap\cw)^\square)=\cl\cap\cw$,
\item[(4)] $\cw$ is accessible and accessibly embedded in $\ck^2$, and
\item[(5)] $|(\cw\times_{\ck^2}\Coalg(L))^\boxplus|=|\cw\times_{\ck^2}\Coalg(L)|^\square$.
\end{enumerate}
\end{propo}
\begin{proof}
Let conditions (1)--(5) be satisfied. Since $\cp=\cw\times_{\ck^2}\Coalg(L)$ is the pullback
\[
\xymatrix@=4pc{
\cw \ar[r]^{} & \ck^2 \\
\cp \ar[u] \ar[r] &
\Coalg(L) \ar[u]
}
\] 
(where $(L,R_0)$ denotes a functorial factorization of $(\cl,\crr_0)$), $\cp$ is accessible and accessibly embedded in $\ck^2$. Assume that $\cp$ is 
$\lambda$-accessible and consider its representative full subcategory $\cj$ of $\lambda$-presentable objects. Let $(L_0,R)$ be a free natural weak 
factorization system over $\cj$ (see \cite{G} 4.4). Following \cite{G} 4.4, this weak factorization is algebraically-free on $\cj$, i.e.,
$$
\overline{\Alg}(R) \xrightarrow{\quad  \overline{\Gamma}\quad} \overline{\Coalg}(L_0)^\boxplus
             \xrightarrow{\quad \quad} \cj^\boxplus
$$
is an isomorphism. Since $\cj$ is colimit-dense in $\cp$, we have $\cj^\boxplus\cong\cp^\boxplus$. Thus
$$
|\cj^\boxplus|=|\cp^\boxplus|=|\cp|^\square=(\cl\cap\cw)^\square 
$$
(see \ref{re3.2}). Thus $(\cl\cap\cw,(\cl\cap\cw)^\square)$ is an accessible weak factorization system.
Hence $(\ck,\cc,\cw)$ is an accessible model category.
\end{proof}

\begin{rem}\label{re5.4}
{
\em
(1)  Condition (5) seems to be too strong. Originally, the author believed that this condition is also sufficient but John Bourke found
a gap in the proof. The author is grateful to him for noticing this.

Assume that $(\ck,\cl,\cw)$ is an accessible model category. Then conditions (1)--(4) are satisfied (see \ref{re5.2}(2)). Let $(\cl_0,\crr)$
be the second accessible weak factorization system of our model category with the functorial factorization given by $L_0'$ and $R'$. We change
this functorial factorization to $L_0$ and $R$ where $L_0f=L_0'Lf$ and $Rf=R_0f\cdot R'Lf$. Since $L_0'Lf\in\cl_0$, $Rf\in\crr$ and 
$f=R_0f\cdot Lf=R_0f\cdot R'Lf\cdot L_0'Lf$, the functors $L_0$ and $R$ provide a functorial factorization of $(\cl_0,\crr)$. There is a functor 
$F_0:\Coalg(L_0)\to \Coalg(L)$ sending an $L_0$-coalgebra $(f,s)$ to the $L$-coalgebra $F_0(f,s)=(f,R'Lf\cdot s)$. This is an $L$-coalgebra because
$R'Lf\cdot s\cdot f=R'Lf\cdot L_0f=R'Lf\cdot L_0'Lf=Lf$ and $R_0f\cdot R'Lf\cdot s=\id$. Let $F:\Coalg(L_0)\to \cw\times_{\ck^2}\Coalg(L)=\cp$ be
the functor induced by $F_0$. 

Objects of $\cp$ are $L$-coalgebras $(f,s)$ such that $f\in\cw$.  
We will show that for any object $(f,r)$ in $\cp$ there exists an $L_0$-coalgebra $(f,s)$ such that $(f,r)=F(f,s)$, i.e., that $F$ is surjective
on objects. Consider the commutative square
$$
\xymatrix@=3pc{
A \ar[r]^{L_0f} \ar[d]_{Lf}& E_0^0(f) \ar[d]^{R'Lf}\\
E_0(f)\ar[r]_{\id} & E_0(f)
}
$$ 
Since $Lf\in\cl_0$ and $R'Lf\in\crr$, there is a diagonal $t$ such that $R'L(f)t=\id$ and $tL(f)=L_0(f)$. Then $(f,tr)$ is an $L_0$-coalgebra
because $R(f)tr=R_0(f)R'L(f)t=R_0(f)$ and $trf=tL(f)=L_0(f)$. We have
$$
F(f,tr)=(f,R'L(f)tr)=(f,r).
$$

But it does not imply (5).
 
(2) In (1), we get a morphism $\xi:(L_0,R)\to (L,R_0)$ of functorial weak factorization systems given by the formula $\xi_f=R'Lf$.
This is assumed in the definition of  an algebraic model category in the sense of \cite{Ri}.
}
\end{rem}

\end{document}